\documentclass[11pt]{amsart}
\usepackage{geometry}                % See geometry.pdf to learn the layout options. There are lots.
\usepackage{graphicx}
\usepackage{amssymb,amsthm}
\usepackage{epstopdf}
\usepackage[all]{xy}
\newtheorem{theorem}{Theorem}[section]
\newtheorem{lemma}[theorem]{Lemma}
\newtheorem{proposition}[theorem]{Proposition}
\newtheorem*{prop*}{Proposition}

\newtheorem*{cor*}{Corollary}
\newtheorem{corollary}[theorem]{Corollary}
\newtheorem*{theorem*}{Theorem}

\theoremstyle{definition}
\newtheorem{definition}[theorem]{Definition}
\newtheorem{example}[theorem]{Example}

\newtheorem{remark}[theorem]{Remark}
\newtheorem*{remark*}{Remark}
\newtheorem*{example*}{Example}

\newcommand{\Z}{{\rm \mathbb Z}}

\newcommand{\secat}{{\sf secat}}
\newcommand{\TC}{{\sf TC}}

\newcommand{\tG}{{\times_G}}
\newcommand{\wTC}{{\sf wTC}}
\newcommand{\cat}{{\sf cat}}
\newcommand{\hdim}{{\sf hdim}}
\newcommand{\Id}{{\rm Id}}

\newcommand{\CP}{\mathbb{CP}}
\newcommand{\C}{{{\mathbb C}}}

\makeatletter
\@namedef{subjclassname@2020}{\textup{2020} Mathematics Subject Classification}
\makeatother

\title{Sequential parametrized topological complexity and related invariants}
\author{M. Farber and J. Oprea}
\thanks{M. Farber was partially supported by a grant from the EPSRC}
\date{26 February 2023}                                           % Activate to display a given date or no date

\begin{document}
\begin{abstract}
Parametrized motion planning algorithms \cite{CFW} have a high degree of universality and flexibility; they generate the 
motion of a robotic system under a variety of external conditions. The latter are viewed as parameters and
constitute part of the input of the algorithm. The concept of sequential parametrized topological complexity 
$\TC_r[p:E\to B]$ is a measure of the complexity of such algorithms. It was studied in \cite{CFW, CFW2} for $r=2$ and in \cite{FP} for $r\ge 2$. 
In this paper we analyse the dependence of the complexity $\TC_r[p:E\to B]$ on an initial bundle with structure group $G$ and on its fibre $X$ 
viewed as a $G$-space. Our main results estimate $\TC_r[p:E\to B]$ in terms of certain invariants of the bundle and the action on the fibre.
Moreover, we also obtain estimates depending on the base and the fibre.
Finally, we develop a calculus of sectional categories featuring a new invariant $\secat_f[p:E\to B]$ which plays an important role in the study of 
sectional category of towers of fibrations.
\end{abstract}


\maketitle
%\section{}
%\subsection{}

\section{Introduction}

 Motion planning algorithms of robotics control autonomous robots in engineering, see \cite{LV}.
 A motion planning algorithm takes as input the initial and the desired states of the system and produces as output a motion of the system starting at the initial states and ending at the desired states.
 A robot is \lq\lq told\rq\rq\ where it needs to go and the execution of this task, including selection of a specific route of motion, is made by the robot itself, i.e. by the robot's motion planning algorithm.
Typically it is understood that the external conditions (such as the positions of the obstacles and the geometry of the enclosing domain) are known and are constant during the motion.

 In papers \cite{CFW, CFW2}, motion planning algorithms of a new type were analyzed.
 These are {\it parametrized motion planning algorithms}, which, besides the initial and desired states, take as input the parameters characterising the external conditions. The output of a parametrized motion planning algorithm is a continuous motion of the system from the initial to the desired state, respecting the given external conditions. The papers \cite{CFW, CFW2} laid out the new formalism and analysed in full detail the problem of moving an arbitrary number $n$ of robots in the domain with $m$ a priori unknown obstacles.
 %The authors used techniques of algebraic topology and were able to find the precise answer by using a combination of upper and lower bounds.

The recent paper \cite{FP} developed a generalisation where the robot must perform {\it a sequence} of tasks. The topological complexity of such an algorithm is called {\it sequential parametrized topological complexity} $\TC_r[p:E\to B]$, where $r=2, 3, \dots$, and the case $r=2$ corresponds to the situations analysed in \cite{CFW, CFW2}.
Formally, $\TC_r[p:E\to B]$ is an integer associated with a fibration $p:E\to B$ where the points of the base $b\in B$ parametrize the external conditions (for example, positions of the obstacles) and for each $b\in B$ the fibre $X_b= p^{-1}(b)\subset E$ is the space of all admissible configurations of the system under the external conditions $b$. To make the present work independent, we include the definition of the concept $\TC_r[p:E\to B]$ and its major properties in \S \ref{sec2}.

In this paper we further analyse  the invariant $\TC_r[p:E\to B]$ trying to understand its dependence on classical invariants of the initial bundle $p:E\to B$: in particular, on its base $B$ and on its fibre $X$. As with all such invariants, exact calculation is generally hard and the development of lower and upper bounds is an essential part of the subject. This is the focus of our work.

We first show that, if the bundle $p:E\to B$ has structure group $G$ with fibre $X$ a $G$-space, then the equivariant sequential topological complexity of $X$ (developed in \cite{CG, BS}) serves as an upper bound for $\TC_r[p:E\to B]$ (see Theorem \ref{upperbound}). The case when $G$ acts freely on $X$ is especially interesting and leads to several somewhat surprising estimates. But
using equivariant topological complexity as an upper bound is fraught with danger since it can be infinite in what appear to be innocuous situations.
%Without the assumption concerning the free action the value of this result is questionable as the estimate does not depend on the base $B$ and hence may rarely be sharp.
%The upper bound of Theorem \ref{upperbound} does not depend on the base $B$ and involves only the fibre $X$ viewed with the action of the structure group $G$ of the bundle; this bound may rarely be sharp since it does not involve information about the base $B$ of the bundle.

As an alternative we develop the notion of \emph{weak sequential equivariant complexity}, denoted by $\TC^w_{r, G}(X)$, and its variant
$\TC^w_{r, G}(X; P)$ which we are tempted (but loathe) to call \emph{weak sequential equivariant complexity with coefficients} $P$ (see \S \ref{sec7}).
We will give several examples showing that these invariants are finite even when equivariant topological complexity is infinite, so they offer the opportunity
for effective estimation in many situations.
Indeed, our main result Theorem \ref{thm:main} gives lower and upper bounds for $\TC_r[p: E\to B]$ in terms of these invariants. To state it one needs to recall the
 invariant
 $$G\mbox{-}\cat[p: E\to B]$$
 introduced by I. M. James  (\cite{Ja}, p. 342). It is defined as the smallest integer $k\ge 0$ such that the base $B$ admits an open cover $B=U_0\cup U_1\cup \dots\cup U_k$ with the property that over each set $U_i$ the bundle $E$ is trivial \emph{as a $G$-bundle}. Clearly, $G\mbox{-}\cat[p: E\to B]$ equals the sectional category $\secat[\tau: P\to B]$ of the associated principal bundle
 that constructs $p\colon E\to B$.
In general,
\begin{eqnarray}\label{gcat1}
G\mbox{-} \cat[p: E\to B]\, \le \, \cat(B)\, \le \dim B
\end{eqnarray}
and
if the group $G$ is $2$-connected (as is the case for simply connected compact Lie groups for instance)
then we can say
\begin{eqnarray}
G\mbox{-}\cat[p: E\to B]\, \le \, \left\lceil \frac{\dim B -3}{4}\right\rceil
\end{eqnarray}
as follows by applying Theorem 5 from \cite{Sch} to $\secat[\tau: P\to B]$.
If the structure group $G$ is discrete then instead of (\ref{gcat1}) one has a stronger inequality
\begin{eqnarray}\label{gcat2}
G\mbox{-} \cat[p: E\to B]\, \le \, \cat_1(B),
\end{eqnarray}
where $\cat_1(B)$ is the sectional category of the universal cover $\tilde B \to B$.
Our main result (Theorem \ref{thm:main}) then is the following.
\begin{theorem*}
For a locally trivial bundle $p: E=X\times_G P\to B=P/G$ one has the inequalities
\begin{eqnarray}\label{ineqmainn}
\TC_{r, G}^w(X; P) \, \le \, \TC_r[p: E\to B] \, \le \, G\mbox{-}\cat[p: E\to B] \, +\, \TC_{r, G}^w(X).
\end{eqnarray}
\end{theorem*}
Note that the first summand in the right hand side (RHS) of (\ref{ineqmainn}) is independent of $r$ and is bounded above by the Lusternik - Schnirelmann category of the base $B$; the second term is the weak equivariant sequential topological complexity of the fibre $X$. In our view, this estimate gets at the heart of the matter. After all, what \emph{is} a bundle? It is just a principal bundle together with
an action of the structure group on the fibre and our upper bound is expressed exactly in numerical quantities derived from these objects. 
In Example \ref{exam:canonical} we shall see that the RHS can be
an equality, so at least in some cases the upper bound can be sharp. Such \emph{a posteriori} knowledge then warrants a deeper study of the invariants  $\TC_{r, G}^w(X)$ and 
$\TC_{r, G}^w(X;P)$ and we hope the present work elicits this.

Beyond defining and applying  the new invariants $\TC_{r, G}^w(X)$ and $\TC_{r, G}^w(X;P)$, in sections \S\S \ref{sec4}, \ref{sec5}, \ref{sec6} 
we develop a calculus of sectional categories, including a new notion denoted by $\secat_f[p: E\to B]$ where $f: B\to C$ is a continuous map. 
It is this notion that allows us to estimate the sectional category of towers of fibrations which serves as the crucial technical tool in the proof of our main results. 
We believe that $\secat_f[p: E\to B]$ holds independent interest and should find application in many situations orbiting the twin galaxies of 
Lusternik-Schnirelmann category and topological complexity.

It is a pleasure to thank  Amit Paul, Debasis Sen and the referee
for several useful comments.

\section{The concept of sequential parametrized topological complexity}\label{sec2}
	
In this section we recall the notion of sequential parametrized topological complexity introduced in \cite{FP}. It is a generalisation of the concept of topological complexity \cite{Fa03} and its parametrized version \cite{CFW}.

Let $p : E \to B$ be a Hurewicz fibration with fibre $X$. Fix an integer $r\ge 2$ and denote
$$E^r_B= \{(e_1, \cdots, e_r)\in E^r; \, p(e_1)=\cdots = p(e_r)\}.$$
The symbol $I=[0,1]$ denotes the unit interval. Let $E^I_B\subset E^I$ be
the space of all paths $\gamma: I\to E$ such that the path $p\circ \gamma: I\to B$ is constant. Fix $r$ points
$$0\le t_1<t_2<\dots <t_r\le 1$$ and
consider the evaluation map
\begin{eqnarray}\label{Pir}
\Pi_r : E^I_B \to E^r_B, \quad \Pi_r(\gamma) = (\gamma(t_1), \gamma(t_2), \dots,  \gamma(t_r)).\end{eqnarray}
$\Pi_r$ is a Hurewicz fibration, see \cite[Appendix]{CFW2}. The  fibre of $\Pi_r$ is $(\Omega X)^{r-1}$, the Cartesian $(r-1)$-th power of the based loop space $\Omega X$.
A section $s: E^r_B \to E^I_B$ of the fibration $\Pi_r$ can be interpreted as {\it a parametrized motion planning algorithm}, i.e.
a function which assigns to every sequence of points $(e_1, e_2, \dots, e_r)\in E^r_B$ a continuous path $\gamma: I\to E$ (representing motion of the system) satisfying $\gamma(t_i)=e_i$ for every $i=1, 2, \dots, r$ and such that the path
$p\circ \gamma: I \to B$ is constant. The latter condition means that the system moves under {\it constant external conditions}
(such as positions of the obstacles etc).

Typically, the fibration $\Pi_r$ does not admit continuous sections, see Corollary 1 and Lemma 1 in \cite{FW2022} which deal with the case $r=2$; when $r>2$ the arguments are similar. Therefore the
motion planning algorithms are necessarily discontinuous in most situations.

The following definition \cite{FP} gives a measure of complexity of sequential parametrized motion planning algorithms.

\begin{definition}\label{def:main}
The {\it $r$-th sequential parametrized topological complexity} of the fibration $p : E \to B$, denoted $\TC_r[p : E \to B]$, is defined as the sectional category of the fibration $\Pi_r$, i.e.
\begin{eqnarray}
\TC_r[p : E \to B]:=\secat(\Pi_r).
\end{eqnarray}
\end{definition}	

In more detail, $\TC_r[p : E \to B]$ is the minimal integer $k$ such that there is an open cover
$\{U_0, U_1, \dots, U_k\}$ of $E^r_B$ with the property that each open set $U_i$
admits a continuous section $s_i : U_i \to E^I_B$ of $\Pi_r$.

Under some mild assumptions instead of open covers one can consider totally general partitions:

\begin{proposition} [see \cite{FP}, Proposition 3.6]
Let $E$ and $B$ be metrizable separable ANRs and let $p: E \to B$ be a locally trivial fibration. Then the sequential parametrized topological complexity $\TC_r[p: E \to B]$ equals the smallest integer $n$ such that $E_B^r$ admits a partition $$E_B^r=F_0 \sqcup F_1 \sqcup ... \sqcup F_n, \quad F_i\cap F_j= \emptyset \text{ \ for } i\neq j,$$
with the property that on each set $F_i$ there exists a continuous section $s_i : F_i \to E_B^I$ of $\Pi_r$.
\end{proposition}

If two fibrations $p : E \to B$ and $p' : E' \to B$ are fibrewise homotopy equivalent then
$\TC_r[p: E \to B] = \TC_r[p': E' \to B],$
see \cite{FP}, Corollary 4.2.

The following upper bound is a reformulation of Proposition 6.1 from \cite{FP}:

\begin{proposition}
	Let $p: E \to B$ be a locally trivial fibration with fiber $X$, where $E, B, X$ are CW-complexes. Assume that the fiber $X$ is $k$-connected, where $k\ge 0$. Then
	\begin{eqnarray}\label{upper}
	\TC_{r}[p: E \to B]\le \left\lceil
	\frac{ r\dim X+\dim B-k}{1+k}\right\rceil .
	%<\frac{\hdim (E_B^r) + 1}{k+1}\leq \frac{r\cdot \dim X+\dim B + 1}{k+1}.
	\end{eqnarray}
\end{proposition}

We refer the reader to \cite{FP} for proofs and further detail.

\section{Relation with the equivariant sequential topological complexity}\label{sec3}

In this section we show that $\TC_r[p:E\to B]$ admits as an upper bound the sequential equivariant topological complexity \cite{CG, BS} of the fibre $X$. This leads to simple estimates in terms of the dimension of the fibre in the case when the structure group $G$ of the fibration acts freely on $X$, see Lemma \ref{lm:26} and Corollary \ref{cor:free}.

\subsection{Equivariant topological complexity} We shall recall a sequential analogue of the notion of equivariant topological complexity introduced by M. Bayeh and S. Sarkar \cite{BS}; it generalizes the concept of equivariant topological complexity originally introduced and studied by H. Colman and M. Grant \cite{CG}.

\subsection{} Let $G$ be a topological group acting on a topological space $X$ from the left. The papers \cite{CG, BS}
require $G$ to be compact but we do not impose this assumption at this stage.

The symbol $X^I$ denotes the space of all continuous paths $\gamma: I\to X$ where $I=[0,1]$ equipped with the compact-open topology. The group $G$ acts naturally on $X^I$, where $(g\gamma)(t) =g\gamma(t)$, for $t\in I$.

For an integer $r\ge 2$ and consider the Cartesian power $X^r= X\times X\times \dots \times X$ ($r$ times). We shall consider the diagonal action of $G$ on $X^r$.

Fix $r$ points $0 = t_1<t_2<\dots <t_r = 1$ in the unit interval $I=[0,1]$ and consider the evaluation map
\begin{eqnarray}\label{rhor}
\rho_r: X^I \to X^r,\end{eqnarray}
where $\rho_r(\gamma) = (\gamma(t_1), \dots, \gamma(t_r))$.
Clearly, $\rho_r$ a $G$-equivariant map.

\begin{definition}\label{def:tcr}
For a path-connected $G$-space $X$, we denote by $\TC_{r, G}(X)$ the smallest integer $k\ge 0$ such that the Cartesian power $X^r= X\times X\times \dots \times X$ ($r$ times) admits an open cover
$X^r=U_0\cup U_1\cup \dots \cup U_k$ with the following properties: each set $U_i$ is $G$-invariant and admits a continuous $G$-equivariant section $s_i: U_i\to X^I$ of the fibration $\rho_r$. If no such cover exists we set
$\TC_{r, G}(X)=\infty.$
\end{definition}

The invariant $\TC_{2, G}(X)$ coincides with the equivariant topological complexity $\TC_G(X)$ of Colman and Grant \cite{CG}.

It is obvious from Definition \ref{def:tcr} that
$$\TC_r(X) \le \TC_{r, G}(X),$$ where $\TC_r(X)$ is the sequential topological complexity of $X$ introduced by Rudyak
\cite{Ru}.

An alternative definition of $\TC_{r, G}(X)$ is obtained as follows (compare \cite{FP}, Lemma 3.5).
Let $K$ be a path-connected locally compact metrizable space and let $k_1, k_2, \dots, k_r\in K$ be a set of $r$ pairwise distinct points.
Consider the set $X^K$ of continuous maps $\alpha: K\to X$ equipped with compact-open topology. The evaluation map
$$\rho^K_r: X^K\to X^r,$$ where $\Pi^K_r(\alpha) = (\alpha(k_1), \dots, \alpha(k_r))\in X^r$, is continuous and $G$-equivariant,
where we view $X^r$ with the diagonal action of $G$.
\begin{lemma} For any path-connected locally compact metrizable space $K$,
the number $\TC_{r, G}(X)$ equals the smallest integer $k\ge 0$ such that the Cartesian power $X^r$
admits an open cover
$X^r=U_0\cup U_1\cup \dots \cup U_k$ with the following properties: each set $U_i$ is $G$-invariant and admits a continuous $G$-equivariant section $s_i: U_i\to X^K$ of $\rho^K_r$.
\end{lemma}
\begin{proof} Consider the commutative diagram
$$
\xymatrix{
X^I \ar@<1ex>[rr]^-{F'}  \ar[dr]_{\rho_r} && X^K \ar@<-1pt>[ll]^-{F}\ar[dl]^{\rho^K_r}\\
 &X^r
}
$$
where the maps $F: X^K\to X^I$ and $F': X^I\to X^K$ are defined as follows. Fix a path $\gamma: I\to K$ satisfying $\gamma(t_i)=k_i$ for all $i=1, \dots, r$. Then $F(\alpha)= \alpha\circ \gamma: I\to X$, where $\alpha\in X^K$.

To define the map $F': X^I\to X^K$ we first construct a continuous function $f: K\to I$ satisfying $f(k_i)=t_i$ for all $i=1, \dots, r$. Applying Tietze extension theorem we find continuous functions $\psi_j: K\to [0,1]$ with $\psi_j(t_i)=\delta_{ij}$ where
$j=1, \dots, r$. Then the function $f=\min\{1, \sum_{i=1}^r t_i\psi_r\}$, $f: K\to I$, has the required properties. The map $F': X^I\to X^K$ is defined by $F'(\alpha)=\alpha\circ f$ where $\alpha\in X^I$.

Clearly the maps $F$ and $F'$ are $G$-equivariant. For an open $G$-invariant subset $U\subset X^r$ any $G$-equivariant section
$s: U\to X^I$ of $\rho_r$ defines the $G$-equivariant section $s'=F'\circ s: U\to X^K$ of $\rho^K_r$. And vice versa, any
$G$-equivariant section $s': U\to X^K$ defines $s=F\circ s': U\to X^I$, an equivariant section of $\rho_r$.

This completes the proof.
\end{proof}

Yet another equivalent characterization of $\TC_{r, G}(X)$ is given by the following (see \cite{BS}):
\begin{lemma} \label{lm:diag}
For a $G$-space $X$ and $r\ge 2$ the integer $\TC_{r, G}(X)$ equals the smallest $k\ge 0$ such that $X^r$ admits an open cover $X^r=U_0\cup U_1\cup \dots\cup U_k$ by $G$-invariant open sets $U_i$ with the property that each inclusion $U_j\to X^r$ is $G$-homotopic to a map with values in the diagonal $X\subset X^r$.
\end{lemma}

Now we can state our result relating the sequential parametrised topological complexity of a fibration with the equaivariant sequential topological complexity of the fibre:

\begin{theorem} \label{thm:equiv} Consider a locally trivial bundle $p: E\to B$ with path-connected fibre $X$ and structure group $G$. Let
$\tau: P\to B$ be a $G$-principal bundle such that $p: E\to B$ coincides with the associated bundle $p: E= X\times_GP = (X\times P)/G \to P/G=B$. Then the sequential parametrized topological complexity $\TC_r[p:E\to B]$ is bounded above by $\TC_{r, G}(X)$, i.e.
\begin{eqnarray}\label{upperbound}
\TC_r[p:E\to B]\, \le \, \TC_{r, G}(X).
\end{eqnarray}
\end{theorem}
Note that the RHS of inequality (\ref{upperbound}) depends only on the fibre $X$ viewed as a $G$-space where
$G$ is the structure group of the bundle.
\begin{proof} First we note that there exists the commutative diagram
$$\xymatrix{
X^I\times_G P\ar[r]^{{} \hskip0.5cm \alpha}\ar[d]_{\rho_r\times_G1}& E^I_B\ar[d]^{\Pi_r}
\\
X^r\times_G P \ar[r]_{\hskip0.4cm \beta}& E^r_B
}$$
where $\alpha$ and $\beta$ are homeomorphisms. Therefore,
$$\TC_r[p:E\to B] = \secat[\Pi_r: E^I_B \to E^r_B]= \secat[\rho_r\times_G1: X^I\times_G P \to X^r\times_G P].$$

For $k=\TC_{r, G}(X)$ let
$X^r=U_0\cup U_1\cup \dots \cup U_k$ be an open cover as in Definition \ref{def:tcr}.
Consider the sets $$W_i= (U_i\times P)/G\subset (X^r\times P)/G.$$
They are open and cover $(X^r\times P)/G$. Any $G$-equivariant section $s_i: U_i\to X^I$ of the fibration $\rho_r$
obviously defines the section
$\sigma_i: W_i \to (X^I \times P)/G$ of the orbit spaces; here $\sigma_i$ is the map induced by $s_i\times 1_P$ on the spaces of orbits. This shows that
$$\TC_r[p:E\to B] = \secat[\rho_r\times_G1: X^I\times_G P \to X^r\times_G P]\le k.$$
\end{proof}
As mentioned in \cite{CG, BS}, in some cases the number $\TC_{r, G}(X)$ is infinite. In particular, one has $\TC_{r, G}(X)=\infty$ if for a subgroup $H\subset G$ the fixed-point set $X^H$ is not path-connected. In such situations the upper bound (\ref{upperbound}) becomes meaningless. We discuss below situations when the number $\TC_{r, G}(X)$ is finite and admits useful upper bounds.

The following Lemma uses the notion of $G$-equivariant homotopy lifting property ($G$-HLP) applied to a map $q: X\to X/G$. This property
means that the commutative diagram
$$
\xymatrix{
Y\ar[r]^f \ar[d]_{\rm {inc}} &X\ar[d]_q\\
Y\times I\ar[r]_F& X/G,
}
$$
where $X$ and $Y$ are separable metric spaces and the map $f: Y\to X$ is $G$-equivariant, can be completed by a $G$-equivariant map
$H: Y\times I\to X$ extending $f$ and such that $q\circ H=F$. A theorem of R. Palais (see \cite{Bre}, Theorem II.7.3) states that
this property is automatically satisfied for free actions of compact Lie groups.
%\begin{lemma}\label{lm:ghlp}
%If $q: X\to X/G$ has $G$-HLP then for any $r\ge 2$ the quotient map $q_r: X^r\to X^r/G$ also has the $G$-HLP.
%\end{lemma}
%\begin{proof}???
%\end{proof}

\begin{lemma} \label{lm:26} Consider a locally trivial bundle $p: E\to B$ with fibre $X$ (a path-connected separable metric space) and structure group $G$.
Assume that the group $G$ acts freely on $X$ and, moreover, that the quotient map $q_r: X^r\to X^r/G$ possesses the
$G$-HLP. Then
$$\TC_r[p:E\to B]\, \le \, \TC_{r, G}(X)\le \cat(X^r/G)\le \dim(X^r/G).$$
\end{lemma}
\begin{proof} In view of Theorem \ref{thm:equiv} we only need to prove the inequality
$\TC_{r, G}(X)\le \cat(X^r/G)$. Consider an open covering $X^r/G = V_0\cup V_1\cup \dots\cup V_k$, where
$k=\cat(X^r/G)$ and each inclusion $U_i\subset X^r/G$ is homotopic to the constant map into a point
$x_0\in X/G\subset X^r/G$; here
$X\subset X^r$ is the diagonal. By our assumption the projection $q_r:X^r\to X^r/G$ has
 the $G$-homotopy lifting property.
The sets $U_i=q_r^{-1}(V_i)\subset X^r$ are $G$-invariant, where $i=0, 1, \dots, k$, and applying the $G$-homotopy lifting property to the homotopy of $V_i$ to $x_0$ we find a homotopy $h_t^i: U_i\to X^r$ (where $t\in [0,1]$ and $i\in \{0,1, \dots, r\}$) such that
$h_0^i$ is the inclusion  $U_i\to X^r$, each map $h^i_t$ is $G$-equivariant and $h_1^i(U_i)\subset X\subset X^r$.
Applying Lemma \ref{lm:diag} we obtain $\TC_{r, G}(X)\le k$.
\end{proof}

%%\subsection{Examples}
%%\begin{example}\label{ex1}
%Suppose that $G$ is a compact topological group acting freely on a connected locally finite CW-complex $X$. Then (see \cite{CG,BS})\marginpar{assumptions on $X$?}
%\begin{eqnarray}\nonumber
%\TC_{r, G}(X) &\le& \cat(X^r/G)\nonumber\\
%&\le & \dim(X^r/G) \nonumber\\
%&= &  r\dim X - \dim G.\nonumber
%\end{eqnarray}%\marginpar{this needs to be justified?}
%%\end{example}%\marginpar{the last line is definitely wrong}
%This remark and Theorem \ref{thm:equiv} lead to the following upper bound:

\begin{corollary}\label{cor:free}
Consider a locally trivial bundle $p: E\to B$ with fibre $X$ (which is a path-connected separable metric space) and a compact Lie group $G$ acting freely on $X$, as the structure group. Then for any $r\ge 2$ one has
\begin{eqnarray}
\TC_{r}[p: E\to B] \, \le \, \cat(X^r/G) \, \le\, r\dim X - \dim G.
\end{eqnarray}
\end{corollary}
\begin{proof} First we note that due to the theorem of Palais (\cite{Bre}, Theorem II.7.3) the assumptions of Lemma \ref{lm:26} are satisfied. We are only left to note that $\dim X^r/G= \dim X^r-\dim G\le r\dim X -\dim G$, see \cite{Pal}, Corollary 1.7.32.
\end{proof}

One can use Lemma \ref{lm:26}  to
give an alternative proof of Proposition 3.3
from \cite{FP} (see also \cite{CFW}, Proposition 4.3) with some minor additional assumptions:

\begin{corollary}
Let $G \to P \stackrel{\tau}{\to} B$ be a principal bundle, where $G$ is a path-connected topological group which has the topology of a separable metric space. Then
$$\TC_r[\tau\colon P \to B] = \cat(G^{r-1}) \quad \mbox{for any}\quad r\ge 2.$$
\end{corollary}

\begin{proof}
By \S 3 of \cite{FP} we know that
$\TC_r[\tau\colon P\to B] \ge \TC_r(G) = \cat(G^{r-1})$.
We view the fibre $G$ as acting on itself by left translations and acting diagonally on $G^r$.
The quotient map $q_r: G^r\to G^r/G$ admits a section
$s: G^r/G \to G^r$, given by $s(g_1, g_2, \dots, g_r) = (e, g_1^{-1}g_2, g_1^{-1}g_3, \dots g_1^{-1}g_r)$. Therefore,
we explicitly obtain a $G$-homeomorphism $G^r \cong G^r/G \times G$ so that   
$q_r$ is a trivial bundle $G^r/G \times G \to G^r/G$ and, as such, has the $G$-HLP.
Lemma \ref{lm:26} then applies and gives the upper bound
$\TC_r[\tau\colon P\to B] \leq \cat(G^r/G)=\cat(G^{r-1})$. Comparing, we see that both
bounds are in fact equalities.
\end{proof}

%
%\vskip 2cm
%
%1. Example when $\TC_{r, G}(X)=\TC_r(X)$.
%
%2. The case of spheres with respect to an involution.
%
%3. What if $G$ is a finite group acting freely on $X$? Can we say what is $\TC_{r, G}(X)$?
%
%4. Example when $\TC_{r, G}(X)=\infty$.
%
%5. Cohomological upper bound for $\TC_{r, G}(X)$ (as Theorem 5.15 in Colman - Grant.
%
%6. More examples.
%

\section{Calculus of sectional categories}\label{sec4}

In this section we introduce a new invariant $\secat_{f}[p:E\to B]$ which generalises the concept of sectional category of a fibration.
This invariant plays a role in estimating sectional category of towers of fibrations, see Theorem \ref{lm:tower}.

\subsection{} Let $p: E\to B$ be a fibration and let $f: B\to C$ be a continuous map.
\begin{definition}\label{def1} We define the invariant
$$\secat_{f}[p:E\to B]$$ to be the smallest integer $k\ge 0$ such that $C$ admits a family of open subsets $U_0, U_1, \dots, U_k$ with the following properties:

(a) $U_0\cup U_1\cup \dots\cup U_k \supset f(B)$ or, equivalently,  $B = \cup_{i=1}^k f^{-1}(U_i)$;

(b) the fibration $p:E\to B$ admits a continuous section over each open set $f^{-1}(U_i)$, where $i=0, 1, \dots, k$.

We set
$\secat_{f}[p:E\to B]=\infty$ if no such family exists.
\end{definition}

Open sets of the form $f^{-1}(U)\subset B$ where $U\subset C$, can be called {\it $f$-saturated}. Definition \ref{def1} can be rephrased as
dealing with covers of the base $B$ by $f$-saturated open sets admitting continuous section of the fibration $p:E\to B$.

%Let $\mathcal T=\mathcal T_B$ be a sub-topology of the topology on $B$; in other words, $\mathcal T$ is a family of open sets of $B$ such that any union of the sets of $\mathcal T$ as well as any finite intersection belong to $\mathcal T$.
%
%\begin{definition}
%The number $$\secat_{\mathcal T}[p:E\to B]$$ is defined as the smallest integer $k\ge 0$ such that there exists an open cover
%$$U_0\cup U_1\cup \dots\cup U_k=C$$ with the property that each $U_i$ belongs to $\mathcal T$ and admits a continuous section $s_i:U_i\to E$, $p\circ s_i=[{\rm {inclusion}}: U_i\to B]$.
%\end{definition}
%
%\begin{example} Let $f: B\to C$ be a continuous map. We can define the family $\mathcal T_f$ to be the set of pre-images of the open sets of $C$. Then the notation
%$$\secat_{\mathcal T_f}[p:E\to B]$$
%can be abbreviated to
%$$\secat_{f}[p:E\to B].$$
%Explicitly, $\secat_{f}[p:E\to B]$ is the smallest integer $k\ge 0$ such that $C$ admits an open cover $C=U_0\cup\dots\cup U_k$ such that the fibration $p:E\to B$ admits a continuous section over each set $f^{-1}(U_i)$, where $i=0, 1, \dots, k$.
%\end{example}

\subsection{Finiteness} The following Lemma summarises information about finiteness of  the invariant $\secat_{f}[p:E\to B]$.

\begin{lemma} Let $p: E\to B$ be a fibration and let $f: B\to C$ be a continuous map.
\begin{enumerate}
\item[(A)] If for some $x\in f(B)\subset C$ one has $\secat[p: p^{-1}f^{-1}(x)\to f^{-1}(x)]>0$ then $\secat_{f}[p:E\to B]=\infty$.
\item[(B)] If $B$ is compact and every point $x\in f(B)\subset C$ has an open neighbourhood $U\subset C$ such that
$\secat[p: p^{-1}f^{-1}(U)\to f^{-1}(U)]=0$ then $\secat_{f}[p:E\to B]<\infty$.
%\item[(C)] In particular, if $B$ is compact, $f(B)=C$, and $f: B\to C$ is a locally trivial bundle such that for every
%$x\in  C$ one has $\secat[p: p^{-1}f^{-1}(x)\to f^{-1}(x)]=0$ then  $\secat_{f}[p:E\to B]<\infty$.
\end{enumerate}
\end{lemma}
\begin{proof} Under the assumption (A) there is no open set $U\subset C$ containing $x$ with $f^{-1}(U)$ having a continuous section. The statement (B) is obvious. 
\end{proof}

In our applications we shall typically have the map $f: B\to C$ surjective, and more specifically, it will often be the quotient map with respect to a group action. However, it is convenient to make no additional assumptions at this stage.

\subsection{Dependence on $f$}
In the special case when the map $f: B\to C=B$ is the identity map, the number $\secat_{f}[p:E\to B]$ turns into the usual sectional category $\secat[p: E\to B]$.  In general, obviously,
\begin{eqnarray}\label{comp1}
\secat[p:E\to B] \le \secat_{f}[p:E\to B]\end{eqnarray}
and
\begin{eqnarray}\label{comp2}
\secat[p:E\to B] = \secat_{f}[p:E\to B]\end{eqnarray}
assuming that $\secat[p:E\to B]=0$.

Moreover, for $B\stackrel f\to  C\stackrel g \to C'$ one clearly has
\begin{eqnarray}\label{compos}
\secat_{f}[p:E\to B]\le \secat_{gf}[p:E\to B].
\end{eqnarray}

%
%Next we consider the effect of a variation of the map $f$:

\begin{lemma}\label{lm:homeo}
Let $p: E\to B$ be a fibration and let $f: B\to C$ and $f': B\to C'$ be two continuous maps. (a) If there is a continuous map
$h: C\to C'$ such that $f'= h\circ f$ then
$$
\secat_f[p:E\to B] \le \secat_{f'}[p: E\to B].
$$
Moreover, (b) If $h: C\to C'$ as above is such that its restriction induces a homeomorphism $f(B)\to f'(B)$ then
$$
\secat_f[p: E\to B] = \secat_{f'}[p: E\to B].
$$
\end{lemma}
\begin{proof}
The statement (a) follows from inequality (\ref{compos}). To prove (b) assume that $U\subset C$ is an open subset with the property that $f^{-1}(U)$ admits a section of $p$. Then $$h(U\cap f(B))=U'\subset f'(B)$$ is an open subset of $f'(B)$ and hence there exists an open subset $V\subset C'$ with $V\cap f'(B)=U'$. Then, $f'^{-1}(V)= f^{-1}(U)$ admits a section of $p$.
Thus any family of open sets
$U_0\cup U_1\cup \dots\cup U_k\supset f(B)$ such that $f^{-1}(U_i)$ admits a section of $p$  determines a family of open subsets of the same cardinality, $V_0\cup V_1\cup \dots\cup V_k\supset f'(B)$ with the preimages $f'^{-1}(V_j)$ admitting sections of $p$. This shows the inverse inequality
$\secat_f[p: E\to B] \ge  \secat_{f'}[p: E\to B]$.
\end{proof}
%\begin{lemma} Let $B$ be an ENR.
%The number $$\secat_{f}[p:E\to B]$$ equals the smallest $k\ge 0$ such that there exists a sequence of
%$\mathcal T$-closed subsets
%\end{lemma}

\subsection{Induced fibrations}
\begin{lemma}\label{lm:map2}
Assume that a fibration $p:E'\to B'$ is induced from the fibration $p:E\to B$ via the map
$\alpha: B'\to B$ as shown on the diagram
$$\xymatrix{
E'\ar[r]^{\beta}\ar[d]_{p'}& E\ar[d]^{p}&
\\
B' \ar[r]_{\alpha}& B\ar[r]_f &C.
}$$
For $f: B\to C$ set
$f'=f\circ \alpha$. Then $$ \secat_{f'}[p':E'\to B'] \le \secat_f[p:E\to B].$$
\end{lemma}

\begin{proof}
Assuming that there is a continuous section $s\colon f^{-1}(U) \to E$ of $p\colon E\to B$, for $U\subset C$ open,
define $\phi\colon f'^{-1}(U) \to E$ by $\phi=s\circ \alpha$.
Then we have $p\circ \phi = \alpha$ and by the pullback property there is a
continuous map $s': f'^{-1}(U) \to E'$ with $p'\circ s'={\rm inclusion}$, i.e. $s'$ is a section of $p'$. Since $f'(B)\subset f(B)\subset C$, we see that the statement of the Lemma follows.
\end{proof}

%Consider the maps $E\stackrel p{\to} B\stackrel f{\to} C$ as above where $p$ is a
%fibration. Let $B'\subset B$ be a subset.
%Denote $E'=p^{-1}(B')$, $p'=p|_{E'}$, and $f'=f|_{B'}$.
%
%\begin{lemma}
%$ \secat_{f'}[p':E'\to B'] \le \secat_f[p:E\to B]$.
%\end{lemma}
%\begin{proof}
%If $U\subset C$ is an open subset and $f^{-1}(U)$ admits a section of $p$
%then $f'^{-1}(U) = f^{-1}(U)\cap B'$ admits a section of $p'$.
%\end{proof}

\begin{lemma}[Maps of fibrations]\label{lm:map1}
If for two fibrations $p:E\to B$ and $p':E'\to B$ over the same base
$B$ there exists a map $\phi:E\to E'$ such that
the diagram 
\begin{center}
$\xymatrix{
 E  \ar[dr]_{ p} \ar[rr]^{\phi}  &&  E' \ar[dl]^{p'} \\
& B
}$
\end{center}
commutes up to homotopy
then $\secat_f[p':E'\to B]\le \secat_f[p:E\to B]$.
\end{lemma}
\begin{proof}
If $U\subset C$ is such that $p$ admits a continuous section $s$ over $f^{-1}(U)\subset B$ then $p'$
admits a homotopy section $\phi\circ s$ over the same subset. Since $p'$ satisfies the homotopy lifting property, the homotopy 
section can be made a genuine section. 
The statement now follows from the definition.
\end{proof}

%\section{Changing $f$}
%
%Without loss of generality we may assume that $f: B\to C$ is a surjection. Otherwise,
%we may replace it by $f: B\to C'=f(B)$.
%
%A surjection $f: B\to C$ defines an equivalence relation $\simeq_f$ on $B$ and a
%topology on the quotient space
%$B/\simeq_f$.
%
%\begin{lemma}
%If two surjections $f: B\to C$ and $f': B\to C'$ are such that the equivalence
%relations $\simeq_f$ and $\simeq_{f'}$ are identical and the topologies on the
%quotients $B/\simeq_f$ and $B/\simeq_{f'}$ coincide then
%$\secat_f[p:E\to B]=\secat_{f'}[p: E\to B]$ for any $p: E\to B$.
%\end{lemma}

\begin{lemma}\label{lm:46}
 Suppose that for two fibrations $p: E\to B$ and $p': E'\to B'$ there exist continuous maps
 $G, \alpha, \beta, \hat\alpha$ shown on the diagram
\begin{center}
$\xymatrix{
{} & E'\ar[d]^{p'} \ar[r]^G & E\ar[d]^p\\
B \ar[r]_\alpha \ar[d]_f & B'\ar[r]_{\beta}\ar[d]^{f'} & B\\
C \ar[r]_{\hat \alpha} &C' & {}
}$
\end{center}
such that the bottom left square is commutative, the upper right square is homotopy commutative and the map $\beta\circ \alpha: B\to B$ is homotopic to the identity map $\Id_B: B\to B$. Then
$\secat_f[p: E\to B]\le \secat_{f'}[p': E'\to B']$.
\end{lemma}

\begin{proof} Consider the fibration $q:\bar E\to B$ induced by the map $\alpha: B\to B'$ from $p': E'\to B'$. It appears in the commutative diagram
\begin{center}
$\xymatrix{
\overline E\ar[r]^\psi \ar[d]_q & E'\ar[d]^{p'} \\
B \ar[r]_\alpha  & B'.
}$
\end{center}
Using Lemma \ref{lm:homeo} and  
Lemma \ref{lm:map2}
one obtains
\begin{eqnarray}\label{ineq11}
\secat_f[q:\bar E\to B] \le \secat_{\hat \alpha\circ f}[q: \bar E\to B]  \le \secat_{f'}[p': E'\to B'].
\end{eqnarray}
Next we note that 
the diagram 
\begin{center}
$\xymatrix{
\overline E\ar[rr]^{G\circ \psi} \ar[dr]_q && E \ar[dl]^p\\
{}&B&
}$
\end{center}
homotopy commutes: 
$$
p\circ G\circ \psi \simeq \beta\circ p'\circ \psi =\beta\circ \alpha\circ q \simeq q.
$$
Applying Lemma \ref{lm:map1} we obtain the inequality $\secat_f[p: E\to B] \le \secat_f[q:\bar E\to B]$
which together with (\ref{ineq11}) implies $\secat_f[p: E\to B]\le \secat_{f'}[p': E'\to B']$, as claimed.
\end{proof}
\begin{corollary} \label{cor:equal} Assume that in the diagram
\begin{center}
$\xymatrix{
E\ar[r]^F\ar[d]_p & E'\ar[d]^{p'} \ar[r]^G & E\ar[d]^p\\
B \ar[r]_\alpha \ar[d]_f & B'\ar[r]_{\beta}\ar[d]^{f'} & B\ar[d]^f\\
C \ar[r]_{\hat \alpha} &C' \ar[r]_{\hat \beta} & C
}$
\end{center}
the maps $p$ and $p'$ are fibrations, the lower squares are commutative, the upper squares are homotopy commutative and the maps $\alpha$ and $\beta$ are mutually inverse homotopy equivalences. Then 
$$\secat_f[p: E\to B]= \secat_{f'}[p': E'\to B'].$$
%===
%
%
%If in addition to the assumptions of  Lemma \ref{lm:46} the maps $\alpha$ and $\beta$ are mutually inverse homotopy equivalences then 
%
%---
%\begin{equation}\label{diag}
%\xymatrix{
%E\ar[r]^F\ar[d]_p & E'\ar[d]^{p'}  \\
%B \ar[r]_\alpha \ar[d]_f & B'\ar[d]^{f'}\\
%C \ar[r]_{\hat \alpha} &C'
%}
%\end{equation}
%the maps $p$ and $p'$ are fibrations and $\alpha$ is a homotopy equivalence then
%
%---
\end{corollary}
\begin{proof}
This follows from applying Lemma \ref{lm:46} twice: to the diagram of Lemma \ref{lm:46} and to the diagram
\begin{center}
$\xymatrix{
{} & E\ar[d]^{p} \ar[r]^F & E'\ar[d]^{p'}\\
B' \ar[r]_\beta \ar[d]_{f'} & B\ar[r]_{\alpha}\ar[d]^{f} & B'\\
C' \ar[r]_{\hat \beta} &C & {}
}$
\end{center}

\end{proof}
\begin{corollary}\label{cor:48}
Suppose that in the following commutative diagram 
$$\xymatrix{
E \ar[r]^F \ar[d]_-p & E' \ar[d]^-{p'} \\
B \ar[r]^-\alpha \ar[d]_-f & B' \ar[d]^-{f'} \\
C \ar[r]^-{\hat\alpha} & C'
}
$$
the maps $p'$, $f$, $f'$ are fibrations
and  $p$ is the induced fibration.
If $\alpha$ and $\hat\alpha$ are homotopy equivalences then
$$\secat_f[p: E\to B]= \secat_{f'}[p': E'\to B'].$$
\end{corollary}

\begin{proof}
By Lemma \ref{lm:map2} and Lemma \ref{lm:homeo}
we have $\secat_{f'}[p': E'\to B'] \geq \secat_{\hat\alpha f}[p:E\to B]\geq \secat_f[p:E\to B]$
and we must only show the inverse inequality.
Since $f$ and $f'$ are fibrations and $\alpha$ and $\hat\alpha$ are homotopy
equivalences, applying the Proposition on page 53 of \cite{May} we see that there exist homotopy inverses $\beta$
and $\hat \beta$ for $\alpha$ and $\hat\alpha$ respectively such that the following diagram
commutes
$$\xymatrix{
B' \ar[r]^-\beta \ar[d]_-{f'} & B \ar[d]^-f \\
C' \ar[r]^-{\hat \beta} & C.
}
$$
We obtain the commutative diagram
\begin{center}
$\xymatrix{
{} & E\ar[d]^{p} \ar[r]^F & E'\ar[d]^{p'}\\
B' \ar[r]_\beta \ar[d]_{f'} & B\ar[r]_{\alpha}\ar[d]^{f} & B'\\
C' \ar[r]_{\hat \beta} &C & {}
}$
\end{center}
with the composition $\alpha \circ \beta: B'\to B'$ homotopic to the identity map. Lemma \ref{lm:46} now gives
$\secat_{f'}[p': E'\to B'] \le \secat_f[p: E\to B]$. This completes the proof. 
%
%----
%Using the assumption that $p: E\to B$ is induced from $p': E'\to B'$ via $f$, we see that 
%
%
%Consider the map $\tau p'\colon E' \to B$ and note that $\alpha \tau p' \simeq p'$
%since $\alpha \tau \simeq 1_{B'}$. Now because the map $p'$ is a fibration and, by
%hypothesis, the top square of the original diagram is a pullback, we see that
%the top square is also a homotopy pullback. The maps $\tau p'$ and $1_{E'}$ then
%provide a map $\chi\colon E' \to E$ with $p\chi \simeq \tau p'$ and $F \chi
%\simeq 1_{E'}$. (In fact, $\chi$ is a homotopy inverse for the homotopy equivalence
%$F$.) This then means we now have a diagram
%
%$$\xymatrix{
%E' \ar[r]^-\chi \ar[d]_-{p'} & E \ar[d]^-p \\
%B' \ar[r]^-\tau \ar[d]_-{f'} & B \ar[d]^-f \\
%C' \ar[r]^-{\hat\tau} & C
%}
%$$
%where the bottom square is commutative and the top square is homotopy commutative.
%
%Now, let $U \subset C$ be an open subset such that there is a section $s\colon f^{-1}(U) \to E$ with 
%$ps=1_U$. Define a map $s'\colon (f')^{-1}(\sigma^{-1}(U)) \to E'$ by $s'=F s \tau$. 
%Note that this definition makes sense because $(f')^{-1}(\sigma^{-1}(U))
%= \tau^{-1}(f^{-1}(U))$. Then we obtain
%$$p's' = p'Fs\tau = \alpha p s \tau = \alpha\tau \simeq 1_{(f')^{-1}(\sigma^{-1}(U))}.$$
%Therefore $s'$ is a homotopy section and, because $p'$ is a fibration, we may
%find a true section. Hence
%$$\secat_f(p) \geq \secat_{f'}(p').$$
\end{proof}

%\begin{thebibliography}{May}
%
%\bibitem{May}
%J. P. May, 
%\emph{A Concise Course in Algebraic Topology}
%Chicago Lectures in Mathematics. University of Chicago Press, Chicago, IL, 1999
%
%
%\end{thebibliography}

\subsection{Homotopical dimension}
For a topological space $A$ having the homotopy type of a finite-dimensional CW-complex we shall denote by
$\hdim(A)$ {\it the homotopical dimension of $A$}; it is defined as the minimal dimension of a CW-complex homotopy equivalent to $A$.

The following Lemma will be used later in our paper.

\begin{lemma}\label{lm:hdim}
Consider a locally trivial bundle $p: E\to B$ where $E$ and $B$ are separable metric spaces and the base $B$ and the fibre $F$ have the homotopy type of finite-dimensional CW-complexes.
Assume also that the fibre $F$ of $p: E\to B$ has finite covering dimension $\dim F$.
Then
the total space $E$ has the homotopy type of a finite-dimensional CW-complex and, moreover,
\begin{eqnarray}
\hdim (E)\, \le \, \hdim (B) +\dim F.
\end{eqnarray}
\end{lemma}
\begin{proof} Let $g: B'\to B$ be a homotopy equivalence where $B'$ is a CW-complex satisfying $\dim B'=\hdim B$.
Consider the diagram
$$
\xymatrix{
E'\ar[r]^G\ar[d]_{p'} & E\ar[d]^p\\
B' \ar[r]_g & B
}
$$
where $p': E'\to B'$ is the fibration induced by $g$. Clearly $G$ is a homotopy equivalence and $\dim E'\le \dim B'+\dim F$.
By Theorem 5.4.2 from \cite{FPbook} the space
$E'$ has homotopy type of a  CW-complex. Hence we obtain
$$\hdim (E)  = \hdim (E')  \le \dim (E') \le \dim B'+\dim F= \hdim (B)+ \dim F.$$
\end{proof}

\subsection{An upper bound}
The following statement gives a useful upper bound for the invariant $\secat_{f}[p:E\to B]$.
%We assume below that the spaces $E, B, C$ are separable metric spaces.

\begin{proposition}\label{lm:upper} Assume that $E$, $B$, $C$ are separable metric spaces.
Let $p: E\to B$ be a fibration and let $f: B\to C$ be a locally trivial bundle such that:

(a) The space $C$ and the fibre $F_0$ of $f: B\to C$ have the homotopy type of CW-complexes;

(b) the fibre $F_1$
of the fibration $p: E\to B$ is $(k-1)$-connected, where $k\ge 0$;

(c) the fibre $F_0$ of the fibration $f: B\to C$ is $d$-dimensional
where $0\le d\le k$.

\noindent Then one has
\begin{eqnarray}\label{upperb}
\secat_{f}[p: E\to B] \, \le \, \left\lceil \frac{\dim B-k}{1+k - d}\right\rceil.
\end{eqnarray}
\end{proposition}
\begin{proof} First we shall prove the statement under an additional assumption that $C$ is a simplicial complex.
We shall remove this assumption afterwards.
%For a topological space $A$ having the homotopy type of a CW-complex we shall denote by
%$\hdim(A)$ {\it the homotopical dimension of $A$}; it is defined as the minimal dimension of a CW-complex homotopy equivalent to $A$.
%\begin{lemma}
%Consider a locally trivial bundle $p: E\to B$ where $E$ and $B$ are separable metric spaces and the base $B$ has the homotopy type of a finite-dimensional CW-complex. Then the total space $E$ also has the homotopy type of a finite-dimensional CW-complex and
%\begin{eqnarray}
%\hdim E\le \hdim B +\dim F.
%\end{eqnarray}
%\end{lemma}

%We shall use the well-known fact of obstruction theory stating that, since the fibre $F_1$ is $(k-1)$-connected, any subset $B'\subset B$
%with $\hdim (B')\le k$
%admits a continuous section of $p: E\to B$.

Consider the skeleta $C^{(i)}\subset C$ of $C$, where $i=0, 1, \dots.$ We know that for any two integers $0\le i< j$ the complement $C^{(i)}-C^{(j)}$ is homotopy equivalent to a simplicial complex of dimension at most $ i-j-1$, see for example \cite{FGLO}, Corollary 5.3.

We may find a chain of open subsets $U_0\subset U_1\subset U_2\subset \dots$ of $C$ such that each set $U_i$ contains $C^{(i)}$ as a strong deformation retract.

Setting $r=k-d$, consider the skeleta
$$C^{(r)} \subset C^{(2r+1)} \subset C^{(3r+2)} \subset \dots \subset C^{((c+1)r+c)},$$
where $c$ is the smallest integer satisfying $\dim C\le (c+1)r+c$, i.e.
$$c=\left\lceil \frac{\dim C-r}{1+ r}\right\rceil =
\left\lceil \frac{\dim B-k}{1+k - d}\right\rceil.
$$
Each complement
$$ X_i = C^{((i+1)r +i)} - C^{(ir +i-1)}, \quad i=0, 1, \dots c,$$
has the homotopy type
of a simplicial complex of dimension $\le r$.
The open set
$$Y_i = U_{(i+1)r+i}-C^{(ir+i-1)}\subset C$$ deformation retracts onto $X_i$ and therefore
$\hdim (Y_i)\le r$. Applying Lemma \ref{lm:hdim} we obtain
$$\hdim (V_i) \le r+d=k,$$ where
$$V_i =f^{-1}(Y_i)\, \subset \, B, \quad i=0, 1, \dots, c.$$
The fibre $F_1$ of $p: E\to B$ is $(k-1)$-connected, and thus we may apply the well-known result of the obstruction theory
stating that the fibration $p:E\to B$ admits a continuous section over each open set $V_i$, where $i=0, 1, \dots, c$.
Since $B=V_0\cup V_1\cup \dots \cup V_c$,
it shows that $\secat_{f}[p:E\to B]\le c$.
This completes the proof in the case when $C$ is a simplicial complex.

Consider now the general case, i.e. we shall only assume that $C$ has the homotopy type of a CW-complex.
We can find a simplicial complex $C'$ and a homotopy equivalence $\hat\alpha: C'\to C$ (see \cite{FPbook}, Theorem
5.2.1).
Consider the fibration $f' : B'\to C'$ induced by $\hat \alpha$ from $f: B\to C$.
The map $\alpha$ shown on the diagram below is a homotopy equivalence.
$$
\xymatrix{
E'\ar[r]^F\ar[d]_{p'} & E\ar[d]^{p}  \\
B' \ar[r]_\alpha \ar[d]_{f'} & B\ar[d]^{f}\\
C' \ar[r]_{\hat \alpha} &C
}
$$
The map $\alpha$ induces the fibration $p': E'\to B'$.
Applying Corollary \ref{cor:48} we obtain that $$\secat_{f'}[p':E'\to B'] = \secat_{f}[p:E\to B]$$ and hence the upper bound
(\ref{upperb}) which we proved above for $\secat_{f'}[p':E'\to B']$ applies to $\secat_{f}[p:E\to B]$ as well.
\end{proof}
%\begin{remark}
%The assumptions of Proposition \ref{lm:upper} concerning the map $f: B\to C$ are satisfied in each of the following two situations:
%
%(A) If $f: B\to C$ is a locally trivial fibration with fibre of dimension $D$;
%
%(B) If $B$ has a group action and $f: B\to C$ is the quotient map collapsing orbits to single points. In this case $D$ equals the dimension of the group. \marginpar{we need to add assumptions on the group action for this to be true}
%
%In the case when $k=0$ the assumption that the fibre $F$ is $(k-1)$-connected means that it is not empty.
%\end{remark}
\begin{remark}
In \cite{FGLO} an upper bound for topological complexity was derived that made use of an
invariant which was called 
$$\widetilde{\TC}(X)=\widetilde{\secat}(E \stackrel{p}{\to} \bar X \stackrel{q}{\to}
X)$$
there, but which we recognize in fact to be $\secat_q[p: E\to \bar X]$ here. In \cite{FGLO} it was further
shown that $\widetilde{\TC}(X)$ could be identified with the the notion of strongly invariant 
topological complexity $\TC_\pi^*(\widetilde X)$ introduced by A. Dranishnikov \cite{Dra} earlier. 
In \cite{PS}, A. K. Paul and D. Sen extended both the invariant $\widetilde{\TC}(X)$
and the strongly invariant topological complexity to the realm of sequential topological complexity
and proved the analogous identification.
This identification, in some sense, was the genesis of our calculus of sectional categories and 
together with Theorem \ref{thm:equiv} begs the question of exactly how parametrized topological
complexity and various forms of equivariant topological complexity are intertwined, especially
in the case of locally trivial fibre bundles.

%Our work was inspired by A. Dranishnikov who introduced in \cite{Dra}  the notion of strongly invariant topological complexity. To recall this notion, assume that a topological group $G$ acts on a topological space $Y$. One defines {\it the strongly equivariant topological complexity}
%$\TC^\ast_G(Y)$ to be the minimal
%integer $k$ such that there is an open cover of $Y \times  Y$ by $(G \times G)$-invariant sets $U_0, \dots,  U_k$ such that for each
%$i$ there is a $G$-equivariant section $\phi_i: U_i \to Y$  (for the diagonal action on $U_i$) of the path fibration
%$Y^I\to Y\times Y$. If $G$ is a discrete group acting freely on $Y$ then $\TC^\ast_G(Y)$ equals
%$\secat_f[\rho_r: Y^I/G \to Y^2/G]$ where $f: Y^2/G \to Y^2/G\times G$ is the natural quotient map. Hence, the invariant of Dranishnikov is a special case of sectional category $\secat_f[p:E\to B]$. A.K. Paul and D. Sen \cite{PS} used the strongly invariant topological complexity of \cite{Dra} to estimate the topological complexity of the total space of a fibration. 
\end{remark}

\section{Sectional category of towers of fibrations}\label{sec5}
\subsection{}
Consider a tower of fibrations
$$E_r\stackrel{p_r}\to E_{r-1}\stackrel{p_{r-1}}\to E_{r-2} \to \dots \stackrel{p_1}\to E_0$$
and the total fibration
$$p=p_1p_2\dots p_r: E_r \to E_0.$$
We shall assume that all spaces $E_i$ are normal.
\begin{theorem}\label{lm:tower}
The sectional category $\secat[p: E_r\to E_0]$ of the total fibration admits the following lower and upper bounds:
\begin{eqnarray*}\label{ineq1}
\hskip 1cm \secat[p_1: E_1\to E_0]  &\le& \secat[p: E_r\to E_0] \, \\
&\le & \, \secat[p_1: E_1\to E_0]
+\sum_{i=1}^{r-1} \secat_{(p_1p_2\dots p_i)}[p_{i+1}:E_{i+1}\to E_i].
\end{eqnarray*}
Here $p_1p_2\dots p_i:E_i\to E_0$ denotes the composition.
\end{theorem}

Lemma \ref{lm1} below will be used in the proof of Theorem \ref{lm:tower}.
\begin{lemma}\label{lm1}
Let $C$ be a normal space. Consider properties $A_1, A_2, \dots, A_r$ of open subsets of $C$, such that each property $A_i$ is inherited by open subsets and disjoint unions. Assume that for each $i=1, 2, \dots, r$
$C$ admits an open cover consisting of
$n_i+1$ open sets satisfying the property $A_i$. Then $C$ admits an open cover consisting of $N+1$ open sets, where $N=\sum_{i=1}^r n_i$, satisfying all the properties $A_1, \dots, A_r$.
\end{lemma}
\begin{proof} For $r=2$ this statement was proven in \cite{OS} as Lemma 4.3. The case $r>2$ follows from this by induction.
\end{proof}

\begin{proof}[Proof of Theorem \ref{lm:tower}] Since the left inequality in (\ref{ineq1}) is obvious we shall concentrate on the right one and use
Lemma \ref{lm1} to prove it. Consider the following properties $A_1, A_2, \dots, A_r$ of open subsets of $E_0$. We shall say that an open subset $U\subset E_0$ satisfies $A_1$ if $U$ has a continuous section of the fibration $p_1$.
For $2\le i\le r$, we shall say that an open subset $U\subset E_0$ satisfies the property $A_i$ if the open set
$p_{i-1}^{-1}\dots p_2^{-1}p_1^{-1}(U)\subset E_{i-1}$ admits a continuous section of $p_i$. By definition, for any $i=1, 2,
\dots r$, the set $E_0$ admits an open cover of cardinality $\secat_{(p_1p_2\dots p_{i-1})}[p_i: E_i\to E_{i-1}]+1$
with each set satisfying $A_i$. Applying Lemma \ref{lm1}, we obtain that $E_0$ admits an open cover $\{U_j\}$ of cardinality
$\sum_{i=1}^r n_i +1$ such that each set $U_j$ satisfies all the properties $A_1, \dots, A_r$.
This means that there exists a continuous section $s_0:U_j\to E_1$ of $p_1$ and
for any $i=1, 2, \dots, r-1$, there exists a continuous section $$s_i: p_i^{-1}\dots p_2^{-1}p_1^{-1}(U_j)\to E_{i+1}$$ of the fibration $p_i$. Hence, the composition
$$s= s_{r-1}s_{r-2}\dots s_1s_0: U_j \to E_r$$
is a well-defined continuous section of the composition $p=p_1p_2\dots p_r: E_r \to E_0.$ This gives the inequality (\ref{ineq1}).
\end{proof}

For convenience of references,  we state below the special case $r=2$ of Theorem \ref{lm:tower} which we combine with the dimension-connectivity upper bound of Proposition \ref{lm:upper}:

\begin{corollary}\label{cor:43}
Consider a tower of fibrations
$E_2\stackrel {p_2}\to E_1\stackrel {p_1}\to  E_0 $
of separable metric spaces. Assume that $p_1: E_1\to E_0$ is locally trivial.
Then the sectional category $\secat[p: E_2\to E_0]$ of the total bundle
$$p=p_2\circ p_1: E_2\to E_0$$
lies between $\secat[p_1: E_1\to E_0]$ and
\begin{eqnarray}\secat[p_1: E_1\to E_0] +\secat_{p_1}[p_2: E_2\to E_1].\end{eqnarray}
Moreover,
under the following additional assumptions
\begin{itemize}
\item [{ (a)}] The fibre of $p_2: E_2\to E_1$ is $(k-1)$-connected;
\item [{ (b)}] The space $E_0$ and the fibre of $p_1: E_1\to E_0$ have the homotopy type of CW-complexes;
\item [{ (c)}] The fibre of $p_1: E_1\to E_0$ has dimension $\le d$ where $0\le d\le k$.
\end{itemize}
one has
\begin{eqnarray}
\secat_{p_1}[p_2: E_2\to E_1]\le \left\lceil
\frac{\dim E_1 -k}{1+k-d}
\right\rceil
\end{eqnarray}

\end{corollary}

\section{Product inequalities}\label{sec6}
%\subsection{Product inequalities}
Lemma \ref{lm1} distills the main results of \cite{O, Dr1,Dr2,OS}, but for the
product inequalities which we describe below we need more specific information about open covers.

An open cover ${\mathcal W}=\{W_0,\ldots,W_{m+k}\}$ of a space $C$ is an
\emph{$(m+1)$-cover} if every subcollection
$\{W_{j_0}, W_{j_1} , \ldots, W_{j_m}\}$ of $m+1$ sets
from ${\mathcal W}$ also covers $C$.
The following simple observation (see \cite{FGLO} for instance)
is the basis for many arguments in this approach.

\begin{lemma}\label{lem:mplus1cov}
A cover ${\mathcal W}=\{W_0,W_1, \ldots,W_{k+m}\}$ is an $(m+1)$-cover of $C$
if and only if each $x \in C$ is contained in at least $k+1$ sets of
${\mathcal W}$.
\end{lemma}

An open cover can be lengthened to a $(k+1)$-cover,
while retaining certain essential properties of the sets in the cover.

\begin{theorem}[{\cite{O,Dr1}}]\label{thm:extopencov}
Let ${\mathcal U}=\{U_0,\ldots,U_k\}$ be an open cover of a normal space $C$.
Then, for any $m=k,k+1,\ldots,\infty$, there is an open $(k+1)$-cover
of $C$, $\{U_0,\ldots,U_m\}$, extending ${\mathcal U}$ such that for $n > k$,
$U_n$ is a disjoint union of open sets that are subsets of the
$U_j$, $0 \leq j \leq k$.
\end{theorem}
We use these facts to obtain inequalities for product fibrations.

\begin{lemma}[Product Inequality, I]\label{lm2}\label{lm:prod1}
Let $p:E\to B$ and $p':E'\to B'$ be fibrations and let $f: B\to C$ and $f': B'\to C'$
be continuous maps. Assume that the spaces $f(B)$ and $f'(B')$ with topology induced from $C$ and $C'$ respectively
are normal.  Then the sectional category of the product fibration
$$\secat_{f\times f'}[p\times p': E\times E'\to B\times B'] $$ is bounded above by the sum
$$\secat_{f}[p: E\to B]+
\secat_{f'}[p': E'\to B']$$
and it is bounded below by
$$\max\{\secat_{f}[p: E\to B], \,
\secat_{f'}[p': E'\to B']\}.$$
%\begin{align*}
%&\max\{\secat_{f}[p: E\to B],
%\secat_{f'}[p': E'\to B']\} \\
%&\le \secat_{f\times f'}[p\times p': E\times E'\to B\times B'] \\
%&\le  \secat_{f}[p: E\to B]+
%\secat_{f'}[p': E'\to B'].
%\end{align*}
\end{lemma}

\begin{proof} First we deal with the lower bounds. Fix a point $b'_0\in B'$ and embed $B$ into $B\times B'$ via
$b\mapsto (b, b'_0)$; also, embed $C$ into $C\times C'$ via $x\mapsto (x, x_0')$ where $x'_0=f'(b'_0)$.
For an open subset $U\subset C\times C'$, a section of $p\times p'$ over $(f\times f')^{-1}(U)\subset B\times B'$
 determines obviously a section of $p$ over $f^{-1}(U\cap (C\times x'_0))$. This implies the inequality
 $\secat_{f\times f'}[p\times p': E\times E'\to B\times B'] \ge \secat_{f}[p: E\to B].$ Similarly, one obtains
  $\secat_{f\times f'}[p\times p': E\times E'\to B\times B'] \ge \secat_{f'}[p': E'\to B'].$

Now we prove the upper bound.
Let $\secat_{f}[p: E\to B]=k$ be realized by open sets $U_0,\ldots,U_k\subset C$ covering $f(B)\subset C$,
with continuous sections $s_j\colon f^{-1}(U_j)\to E$ of $p$, and let $\secat_{f'}[p': E'\to B']=m$
be realized by open sets $V_0,\ldots, V_m\subset C'$ covering $f'(B')$,
with sections
$s'_j\colon f'^{-1}(V_j)\to E'$  of $p'$.
By Theorem \ref{thm:extopencov} we can extend the family $U_0,\ldots,U_k$ to a family of open subsets
$U_0,\ldots,U_{k+m}$ of $C$ such that any $k+1$ members of this family cover $f(B)$. Similarly,
we can find a family $V_0,\ldots,V_{k+m}$ of open subsets of $C'$ extending the initial family
$V_0,\dots, V_m$
such that
any $m+1$ members of this extended family cover $f'(B')$.
Theorem \ref{thm:extopencov} guarantees that every set of the form $f^{-1}(U_j)$ or $f'^{-1}(V_j)$ admits a continuous section of $p$ or $p'$ respectively, where $j=0, 1, \dots, k+m$.

Denotting $W_j=U_j\times V_j$,  where
$j=0,\ldots,k+m$, we see that each set
$(f\times f')^{-1}(W_j)=f^{-1}(U_j)\times f'^{-1}(V_j)$ admits a continuous section of $p\times p'$.
We show below
that the sets $W_j$ cover $f(B)\times f'(B')$ which implies that $\secat_{f\times f'}[p\times p': E\times E'\to B\times B'] \leq k+m$.

Suppose that a point $(x,y)\in f(B) \times f'(B')$ is not in any of the sets $W_j$, where $j=0,\ldots,k+m$.
Since any $k+1$ sets
$U_j$ cover $f(B)$, we know that $x$ belongs to at least $m+1$ of the $U_j$, by Lemma \ref{lem:mplus1cov}.
Without loss of generality, we may assume that $x\in U_0 \cap U_1 \cap \ldots \cap U_m$. Then
$y \not \in V_0 \cup V_1 \cup \ldots \cup V_m$, in view of our assumption.
Therefore, $y$ can only lie in the sets
$V_{m+1},\ldots, V_{k+m}$ which is a contradiction since $y$ belongs to at least $k+1$ of the
sets $V_j$, by Lemma \ref{lem:mplus1cov}.
\end{proof}

\subsection{} Next we state another product inequality dealing with fibrations over the same base.
\begin{lemma}[Product inequality, II] Let $p: E\to B$ and $p':E'\to B$ be two fibrations, and let $f: B\to C$. We shall assume that $f(B)$ is normal in the topology induced from $C$. Then the sectional category
$$\secat_f[p\times_B p': E\times_B E'\to B]$$ of the fibrewise product is bounded below by $$\max\{\secat_f[p: E\to B],\, \secat_f[p': E'\to B]\}$$ and is bounded above by the sum $$\secat_f[p: E\to B]+\secat_f[p': E'\to B].$$
%
%Form the fibrewise product $p\times_B p': E\times_B E'\to B$.
%Then one has (1)
%\begin{align*}
%\max\{\secat_f[p: E&\to B],\, \secat_f[p': E'\to B]\} \\
%&\leq \secat_f[p\times_B p': E\times_B E'\to B] \\
%&\le \secat_f[p: E\to B]+\secat_f[p': E'\to B].
%\end{align*}
(2) Moreover,
$$\secat_f[p\times_B p': E\times_B E'\to B] = \secat_f[p: E\to B]$$
if $\secat[p': E'\to B]=0$, i.e. if $p'$ admits a section.
\end{lemma}
\begin{proof} The projection $pr: E\times_B E'\to E$ appears in the commutative diagram
$$\xymatrix{
 E \times_B E' \ar[dr]_{ p\times_B p'} \ar[rr]^{pr}  &&  E \ar[dl]^{p} \\
& B
}$$
and Lemma \ref{lm:map1} gives $\secat_f[p\times_B p': E\times_B E'\to B]\ge \secat_f[p: E\to B]$.
Similarly one gets the lower bound using $\secat_f[p': E'\to B]$, which proves the statement concerning the lower bound.
Next we note that
\begin{eqnarray}\label{ineq}
\secat_f[p\times_B p': E\times_B E'\to B]\le \secat_{f\times f}[p\times p': E\times E'\to B\times B].
\end{eqnarray}
Indeed, the fibration $p\times_B p': E\times_B E'\to B$ is induced from the product fibration $p\times p': E\times E'\to B\times B$
by the diagonal map $\Delta: C\to C\times C$.
Lemma \ref{lm:map2} gives the inequality
$\secat_{(f\times f)\circ \Delta}[p\times_B p': E\times_B E'\to B]\le \secat_{f\times f}[p\times p': E\times E'\to B\times B]$. Finally, we can apply Lemma \ref{lm:homeo} and replace $(f\times f)\circ \Delta$ by $f$. Combining (\ref{ineq}) with Lemma \ref{lm:prod1} we obtain the upper bound.

Statement (2) obviously follows by combining the lower and upper bounds.
\end{proof}

\section{Weak equivariant topological complexity $\TC^w_{r,G}(X)$}\label{sec7}

%\subsection{}
Let $p: E\to B$ be a bundle with fibre $X$ and structure group $G$ which is associated
to a principal bundle $\tau: P\to B$. In other words,
$E=X\times _G P.$

As in \S \ref{sec2}, we fix $r\ge 2$ points $0=t_1< t_2<\dots<t_r=1$ and consider the evaluation map
$$\rho_r: X^I\to X^r, \quad \rho_r(\gamma)=(\gamma(t_1), \gamma(t_2),\dots, \gamma(t_r)) \quad \mbox{where}\quad \gamma\in X^I.$$
Consider also the quotient map $$q_r: X^r\to X^r/G,$$ where we view $G$ acting diagonally on $X^r$.

The following invariant plays an important role in our main Theorem \ref{thm:main}:
\begin{eqnarray}\label{wtcr}
\TC^w_{r, G}(X) = \secat_{q_r}[\rho_r: X^I\to X^r].
\end{eqnarray}
Explicitly, we have:

\begin{definition} \label{deftcw} The invariant
$\TC^w_{r, G}(X)$ equals the smallest integer $k\ge 0$ such that $X^r$ admits an open cover $X^r=U_0\cup U_1\cup\dots\cup U_k$ by $G$-invariant open sets such that for each $i=0, 1, \dots, k$ there is a continuous section $s_i: U_i\to X^I$ of $\pi_r$. \end{definition}

Note that the section $s_i$ in Definition \ref{deftcw} is not required to be $G$-equivariant, unlike in the case of $\TC_{r, G}(X)$. This explain the adjective {\it "weak"} and the symbol  {\it \lq\lq w\rq\rq} in the notation.
We obviously have
\begin{eqnarray}\label{compare}
\TC_r(X)\, \le \, \TC^w_{r, G}(X) \, \le \, \TC_{r, G}(X),
\end{eqnarray}
where the left inequality is a special case of (\ref{comp1}).
All these inequalities become equalities when the action of $G$ is trivial.
\begin{lemma} \label{lm:41} For any $G$-space $P$ one has
$$\TC^w_{r, G}(X) = \secat_{q_r\times \epsilon}[\rho_r\times 1: X^I\times P\to X^r\times P]$$
where $\epsilon: P\to \ast$ is the map onto a singleton.
\end{lemma}
\begin{proof}
This follows from Lemma \ref{lm:prod1} since clearly $\secat_\epsilon[1: P\to P]=0$.
\end{proof}
Next we state the dimension-connectivity upper bound:
\begin{lemma}\label{lm:63}
Assume that $X$ is a $k$-connected simplicial complex and $G$ is a topological group
homeomorphic to a CW-complex acting freely on $X$ and such that the map $q_r: X^r\to X^r/G$ is a locally trivial bundle.
If $\dim G\le k$ then
\begin{eqnarray}
\TC^w_{r, G}(X) \le \left\lceil
\frac{r\dim X -k }{1+k-\dim G}
\right\rceil
\end{eqnarray}
\end{lemma}
\begin{proof} We apply Proposition \ref{lm:upper} having in mind that the fibre $(\Omega X)^{r-1}$ of fibration $\rho_r$ is $(k-1)$-connected.
\end{proof}
As a special case of Lemma \ref{lm:63} we mention:
\begin{corollary}
If $X$ is $k$-connected, where $k\ge 0$, and the group $G$ is discrete and the quotient map
$q_r: X^r\to X^r/G$ is a covering map then
\begin{eqnarray}\label{upper3}
\TC^w_{r, G}(X) \le  \, \left\lceil \frac{r \dim X-k}{1+k}\right\rceil
\end{eqnarray}
\end{corollary}

%\subsection{}
We shall be discussing yet another invariant $\TC_{r, G}^w(X; P)$ given by
\begin{eqnarray}\label{tcrgwp}
\TC_{r, G}^w(X; P)= \secat_{Q}[\rho_r\times 1: X^I\times P\to X^r\times P]
\end{eqnarray}
with
$Q: X^r\times P \to X^r\times_G P$ being the natural projection; here $X$ and $P$ are $G$-spaces and $\rho_r$ is the fibration (\ref{rhor}).
Comparing with Lemma \ref{lm:41} we see that it is similar to
$\TC_{r,G}^w(X)$ with the only distinction that the map $q_r\times \epsilon$ is replaced by $Q$.
\begin{lemma} One has
\begin{eqnarray}\label{tcrq}
\TC_r(X) \, \le \, \TC_{r, G}^w(X; P)\, \le \, \TC_{r, G}^w(X).
\end{eqnarray}
\end{lemma}
\begin{proof}
Consider the following commutative diagram
$$\xymatrix{
X^I\times P\ar[r]^{\rho_r\times 1} \ar[d]_{p_1} &X^r\times P\ar[d]_{p_2}\ar[r]^Q & X^r\times_G P\ar[d]^{p_3}\\
X^I \ar[r]_{\rho_r} & X^r\ar[r]_{q_r} & X^r/G
}$$
where the maps $p_1, p_2, p_3$ are projections on the first factor. Since the fibration $\rho_r\times 1$ is induced from $\rho_r$ via $p_2$, we may apply Lemma \ref{lm:map2} to conclude
\begin{eqnarray*}
\TC^w_{r, G}(X) &=& \secat_{q_r}[\rho_r: X^I\to X^r] \\
&\ge&\secat_{q_r\circ p_2}[\rho_r\times 1: X^I\times P\to X^r\times P] \\
&=&\secat_{p_3\circ Q}[\rho_r\times 1: X^I\times P\to X^r\times P] \\
&\ge&\secat_{Q}[\rho_r\times 1: X^I\times P\to X^r\times P]\\
&=& \TC^w_{r, G}(X; P).
\end{eqnarray*}
On the third line we used Lemma \ref{lm:homeo}, part (a). This proves the right inequality in (\ref{tcrq}).
The left inequality follows from
\begin{eqnarray*}
\TC^w_{r, G}(X;P)&=& \secat_{Q}[\rho_r\times 1: X^I\times P\to X^r\times P] \\
&\ge& \secat[\rho_r\times 1: X^I\times P\to X^r\times P]\\
&=&
\secat[\rho_r: X^I\to X^r]\\
&=& \TC_r(X)
\end{eqnarray*}
where on the second line we used inequality (\ref{comp1}) and on the third line Lemma \ref{lm:prod1}.
\end{proof}

Next result gives a dimension-connectivity upper bound for $\TC_{r, G}^w(X;P)$ which holds for weaker assumptions on $X$ compared to Lemma \ref{lm:63}.
\begin{lemma}\label{lm:76}
Assume that $X$ is a $k$-connected simplicial complex and $G$ is a topological group
homeomorphic to a CW-complex. Suppose that $P\to P/G$ is a locally trivial bundle.
If $\dim G\le k$ then
\begin{eqnarray}
\TC^w_{r, G}(X;P) \le \left\lceil
\frac{r\dim X +\dim P -k }{1+k-\dim G}
\right\rceil.
\end{eqnarray}
\end{lemma}
\begin{proof} This follows by applying Proposition \ref{lm:upper} to the definition (\ref{tcrgwp}).
\end{proof}

\begin{example}\label{ex:77} 
Consider the unit circle $S^1\subset \mathbb C$ with the action of the cyclic group of order two $G=\mathbb Z_2$ acting as the complex conjugation, $z\mapsto \bar z$. We know from \cite{CG} that in this case $\TC_{2, G}(S^1)$ is infinite due to the fact that the set of fixed points is disconnected.

On the other hand one can consider the following open cover $S^1\times S^1 = U_0\cup U_1$ where
$U_0=\{(z_1, z_2); z_1\not=-z_2\}$ and $U_1=\{(z_1, z_2); z_1\not=z_2\}$. These sets are $G$-invariant and over each of these sets one has the well known continuous sections. Thus, we obtain that $\TC_{2, G}^w(S^1)= 1$.
\end{example}

\begin{example} Consider the more general case of a sphere $S^n$, where $n\ge 1$, with an action of a discrete group $G$.
First we apply the upper bound (\ref{upper3}) with $k=n-1$ to obtain
$$\TC_{r, G}^w(S^n) \le r \quad \mbox{for any}\quad r\ge 2.$$
Secondly, using (\ref{tcrq}) and the result of Y. Rudyak \cite{Ru} (stating that $\TC_r(S^n)$ equals $r$ for $n$ even and $r-1$ for $n$ odd), we obtain that for any even $n$
\begin{eqnarray}
\TC_{r, G}^w(S^n) = r.
\end{eqnarray}
For $n$ odd our inequalities imply that $\TC_{r, G}^w(S^n)$ equals either $r-1$ or $r$.
\end{example}

\begin{example}\label{exam:s1s2}
Let $S^1$ act on $S^2$ by
rotations about the $z$-axis. The fixed point set of the action
is the disconnected set $\{N,S\}$, where $N$ and $S$ are the North and South Poles
respectively, so the equivariant topological complexity is infinite, $\TC_{r, S^1}(S^2)=\infty$ for all $r\geq 2$.

Let us now examine the weak equivariant topological complexity $\TC_{2,S^1}^w(S^2)$. Fix an orbit $O\subset S^2$ given by the equator and 
fix an orientation of $O$.
Consider the open cover $S^2\times S^2= U_0\cup U_1\cup U_2$
where
\begin{eqnarray*}
U_0&=&\{(x,y)\,|\, x\not = -y\},\\
U_1&=&\{(x,y)\,|\, x\not = y\}-\{(N,S), (S,N)\},\\
U_2&=& \{(x,y)\,|\, x\not\in O\, \mbox{and}\, y\not\in O\}.
\end{eqnarray*}
Clearly, the sets $U_0, U_1, U_2$ are $S^1$-invariant. We may define the motion planning rules over each of the sets $U_i$ as follows.
For $(x,y)\in U_0$, go from $x$ to $y$ along the shortest geodesic arc. For $(x, y)\in U_1$ the point $x$ moves along the shortest geodesic arc first to the closest point of $O$, then along $O$ in the positive direction to the point closest to $y$, and finally to $y$. For $(x, y)\in U_2$ the point $x$ moves along the shortest geodesic arc to the closest Pole ($N$ or $S$), then to the closest Pole to $y$  along a fixed path and then to $y$; the first and the third portions are along the shortest geodesic arc on the sphere $S^2$. Hence $\TC_{2,S^1}^w(S^2) \leq 2$.
Since $2=\TC(S^2) \leq \TC_{2,S^1}^w(S^2)$, we see that $\TC_{2, S^1}^w(S^2)=2$.\end{example}

\section{Bounds for the sequential parametrized topological complexity}

Finally we are in position to state and prove the main result of this paper:

\begin{theorem}\label{thm:main}
Let $p: E\to B$ be a locally trivial fibre bundle with structure group $G$, the fibre $X$ and the associated principal bundle $\tau: P\to B$. Then the sequential parametrized topological
complexity $\TC_r[p: E\to B]$
admits the following upper and lower bounds
\begin{eqnarray}\label{ineqmain}
\TC_{r, G}^w(X; P) \, \le \, \TC_r[p: E\to B] \, \le \, G\mbox{-}\cat[p: E\to B] \, +\, \TC_{r, G}^w(X;P).
\end{eqnarray}
\end{theorem}
\begin{proof} Since $E=X\times_G P$ we have
$$E^r_B = X^r\times_G P \quad \mbox{ and}\quad E^I_B=X^I\times_G P \quad \mbox{for any}\quad r\ge 2.$$
The map $\Pi_r: E^I_B \to E^r_B$ becomes
$$\rho_r\times 1: X^I\times_G P \to X^r\times_G P \quad \mbox{where}\quad \rho_r(\gamma)=(\gamma(t_0), \dots, \gamma(t_r)). $$

Consider the commutative diagram
\begin{eqnarray}\label{diag1}
\xymatrix{
X^I\times P \ar[r]^{Q'}\ar[d]_{\rho_r\times 1} & X^I\times_G P\ar[d]^{\rho_r\times_G 1}\\
X^r\times P \ar[r]_Q & X^r\times_GP
}\end{eqnarray}
where $Q: X^r\times P \to X^r\times_G P$ and $Q': X^I\times P \to X^I\times_G P$
are the natural projections.
Using Lemma \ref{lm:map1} and Theorem \ref{lm:tower} we have
\begin{eqnarray*}{\sf TC}_r[p:E\to B]&=&\secat[\rho_r\times_G 1: X^I\times_G P \to X^r\times_G P]\\
&\le& \secat[(\rho_r\times_G 1)\circ Q': X^I\times P\to X^r\times_G P]\\
&=& \secat[(Q\circ (\rho_r\times 1): X^I\times P\to X^r\times_G P]\\
&\le& \secat[Q: X^r\times P\to X^r\times_G P] +\secat_Q[\rho_r\times 1: X^I\times P\to X^r\times P]
\end{eqnarray*}
 Next we observe that
$$\secat[X^r\times P\to X^r\times_G P]\le \secat[\tau: P\to B]=G\mbox{-}{\sf cat}[p:E\to B]$$
and
$$\secat_Q[X^I\times P\to X^r\times P]=\TC^w_{r, G}(X;P).$$ %\le \TC^w_{r, G}(X).$$
Thus, we obtain the right inequality in (\ref{ineqmain}).

For the left inequality (\ref{ineqmain}) we consider again diagram (\ref{diag1}) and observe that the fibration
$\pi_r\times 1: X^I\times P\to X^r\times P$ is induced from $\pi_r\times_G 1: X^I\times_G P\to X^r\times_G P$ via $Q$. Therefore, using  Lemma \ref{lm:map2} we obtain
\begin{eqnarray*}
\TC_r[p: E\to B]&=& \secat[\rho_r\times 1: X^I\times_G P \to X^r\times_G P] \\
&\ge& \secat_Q[\rho_r\times 1: X^I\times P \to X^r\times P]= \TC^w_{r, G}(X;P).
\end{eqnarray*}
This completes the proof.
\end{proof}
\begin{remark} Due to the right inequality (\ref{tcrq}), the upper bound in (\ref{ineqmain}) gives
\begin{eqnarray}\label{ineq:main2}
\TC_r[p: E\to B] \, \le \, G\mbox{-}\cat[p: E\to B] \, +\, \TC_{r, G}^w(X).\end{eqnarray}
The RHS of this inequality has two terms, one depending only on the initial bundle $p:E\to B$ and the other depending only on the fibre, $X$ viewed as a $G$-space.
\end{remark}

Theorem \ref{thm:main} implies that for the trivial bundle $p:E\to B$ with fibre $X$ one has \newline $\TC_r[p:E\to B]=\TC_r(X)$, see
Example 3.2 from \cite{FP}.
Indeed, in this case $G$-$\cat[p:E\to B]=0$ and $\TC_{r, G}^w(X, P)=\TC_r(X)$ and hence the statement follows from (\ref{ineqmain}).

\begin{example}
The Klein bottle $K$ is the total space of the bundle
$p: K = S^1 \times_{\Z/2} S^1 \to S^1$
with the associated principal bundle the $2$-fold covering
$\tau: S^1 \to S^1$ and the action of $G=\Z/2$ on the fibre $S^1$ being given by
reflection in the last coordinate. The inequality (\ref{ineq:main2}) with $r=2$ and the result of Example \ref{ex:77} give
\begin{eqnarray}\label{klein}
\TC[p: K \to S^1] = \TC_2[p: K \to S^1] \leq 1 + 1 = 2.
\end{eqnarray}

Mark Grant observed that (\ref{klein}) is in fact an equality. The inequality $\TC[p: K \to S^1]\ge 2$ can be obtained by applying Theorem 2 from \cite{FW2022}.  The bundle $p: K\to S^1$ is the unit sphere bundle of a rank 2 vector bundle $\xi$ over the circle $S^1$. One has $\rm w_2(\xi)=0$ (for dimensional reasons) and $\rm w_1(\xi)\not=0$ (since $\xi$ is not orientable) and therefore the relative height 
$\mathfrak h(\rm w_1(\xi)|\rm w_2(\xi))=1$ equals one. Theorem 2 from \cite{FW2022} now applies and gives an equality 
$\TC[p: K \to S^1]=2$. 
\end{example}

\begin{example}\label{exam:canonical}
Consider the principal $G$-bundle $\tau: P\to B$ where $G=S^1$, $P=S^{2n+1}$ and $B=\CP^n$ (the Hopf bundle).
Here the sphere $S^{2n+1}$ is viewed as the unit sphere in $\C^{n+1}$ and the circle $S^1$ acts on it by complex multiplication. Let $X=S^2$ with $S^1$-action given by rotations about the $z$-axis, as in Example \ref{exam:s1s2}.
Consider the fibre bundle $p: E\to B$ with fibre $X=S^2$ where $E=X\times_G P$. Applying (\ref{ineq:main2}) with $r=2$ we obtain
\begin{eqnarray}\label{eq:31}
\TC[p: E\to B] = \TC_2[p: E\to B] \le \secat[\tau: P\to B] + \TC_{2, G}^w(X)\end{eqnarray}
and from Example \ref{exam:s1s2} we know that $ \TC_{2, G}^w(X)=2$.
On the other hand, since $\cat(\CP^n)=n$, we have
$\secat[\tau: P\to B]\leq \cat(B)=n$ (in fact, this is an equality by a cuplength argument). Thus, (\ref{eq:31}) gives
$\TC[p: E\to B] \leq n+2$.

In \cite{FW2022a} the authors studied parametrized topological complexity of sphere bundles. The sphere bundle
$p: E\to B$ which was discussed in the previous paragraph is the unit sphere bundle associated with the rank 3 vector bundle
over $B=\CP^n$ which is the Whitney sum $\eta\oplus \epsilon$ where $\eta$ is the canonical complex line bundle over
$\CP^n$ and $\epsilon$ is a trivial real line bundle. The result of Example 20 from \cite{FW2022a} states that
$\TC[p:E\to B] \le n+2$ and moreover $\TC[p:E\to B] = n+2$ for any even $n$.

Here the point is that, in the example above, the upper bound (\ref{ineq:main2}) is in fact sharp; that is, we have an equality
$$\TC[p\colon E \to B] = G\mbox{-}\cat[p\colon E \to B] + \TC^w_{2,G}(S^2).$$
In fact, since in general $\TC^w_{r,G}(X;P) \leq \TC^w_{r,G}(X)$, 
we see that (\ref{ineqmain}) in this case is an equality as well. This emphasizes the fact that these upper bounds can 
sometimes detect parametrised topological complexity precisely.
\end{example}


\begin{thebibliography}{CLOT}

\bibitem{BS}
M. Bayeh, S. Sarkar,
\emph{Higher equivariant and invariant topological complexities},
J. Homotopy Relat. Struct. 15 (2020), no. 3-4, 397--416.

\bibitem{Bre} G. Bredon, \emph{Introduction to compact transformation groups}, Academic Press, New York, London, 1972.


\bibitem{CFW} D. Cohen, M. Farber, S. Weinberger, \emph{Topology of parametrized motion planning algorithms.} SIAM J. Appl. Algebra Geom. {\bf 5} (2021), no. 2, 229–249.

\bibitem{CFW2} D.C. Cohen, M. Farber, S. Weinberger,
\emph{Parametrised topological complexity of collision-free motion planning in the plane.}
 \lq\lq Annals of Mathematics and Artificial Intelligence\rq\rq,  {\bf 90} (2022), no. 10, 999–1015. 

\bibitem{CG} H. Colman, M. Grant, \emph{Equivariant topological complexity.} Algebr. Geom. Topol. 12 (2012), no. 4, 2299–2316.

\bibitem{Dr1}
A. Dranishnikov,
\emph{On the Lusternik-Schnirelmann category of spaces with
2-dimensional fundamental group},
Proc. Amer. Math. Soc. 137 (2009), no. 4, 1489--1497.

\bibitem{Dr2}
A. Dranishnikov,
\emph{The Lusternik-Schnirelmann category and the fundamental group},
Algebraic \& Geometric Topology {\bf 10} (2010), 917--924

\bibitem{Dra} A. Dranishnikov,
\emph{On topological complexity of twisted products.}
Topology Appl. {\bf 179} (2015), 74–80.

\bibitem{Fa03} M. Farber,
\emph{Topological complexity of motion planning.}
Discrete Comput. Geom. 29 (2003), no. 2, 211–221.

\bibitem{FP} M. Farber, A.K. Paul,
\emph{Sequential parametrized motion planning and its complexity},  Topology Appl. 321 (2022), Paper No. 108256, 23 pp.

\bibitem{FGLO}  M. Farber, M. Grant, G. Lupton, J. Oprea, \emph{An upper bound for topological complexity.} Topology \& Appl. 255 (2019), 109 -- 125.

\bibitem{FW2022} M. Farber, S. Weinberger, \emph{Parametrized motion planning and topological complexity}, Algorithmic Foundations of Robotics XV,  Steven M LaValle et al editors, Springer 2023, pp 1- 17.

\bibitem{FW2022a} M. Farber, S. Weinberger, \emph{Parametrized topological complexity of sphere bundles},
arXiv:2202.05796 to appear in the volume of the journal "Topological methods of nonlinear analysis" dedicated to E. Fadell and S. Husseini.

\bibitem{FPbook}  R. Fritsch, R.A. Piccinini, \emph{Cellular structures in topology}. Cambridge Studies in Advanced Mathematics, 19. Cambridge University Press, Cambridge, 1990. xii+326 pp.

\bibitem{Gra}
M. Grant,
\emph{Parametrised topological complexity of group epimorphisms},  Topol. Methods Nonlinear Anal. 60 (2022), no. 1, 287–303. 

\bibitem{Ja} I.M. James, \emph{On category, in the sense of Lusternik and Schnirelmann}, Topology, {\bf 17}(1978), 331 - 348.

\bibitem{LV} S. M. LaValle, \textit{Planning algorithms}, Cambridge University Press, 2006.

\bibitem{May}
J. P. May, 
\emph{A Concise Course in Algebraic Topology}
Chicago Lectures in Mathematics. University of Chicago Press, Chicago, IL, 1999


\bibitem{OS}
J. Oprea, J. Strom,
\emph{Mixing Categories},
Proc. Amer. Math. Soc.  139 (2011), 3383--3392.

\bibitem{O}
P. Ostrand,
\emph{Dimension of metric spaces and Hilbert's problem $13$}, Bull. Amer. Math. Soc. 71 (1965), 619--622.

\bibitem{PS}
A. Paul, D. Sen,
\emph{An upper bound for higher topological complexity and higher
strongly equivariant complexity}, Topology Appl. 277 (2020) 107172 18pp.

\bibitem{Pal} R. Palais, The classification of $G$-spaces, Mem. Amer. Math. Soc. 36 1960 iv+72 pp.


\bibitem{Ru} Y. B. Rudyak, \emph{On higher analogs of topological complexity,}
Topology and its Applications 157 (2010), 916–920.

 \bibitem{Sch} A. S. Schwartz, \emph{The genus of a fibre space.} Trudy Moscow Mat. Soc. {\bf 11}(1962),  99–126.
\end{thebibliography}
\end{document}